\documentclass[12pt]{amsart}
\usepackage{amsmath,latexsym,amsfonts,amssymb,amsthm}
\usepackage{geometry}
\usepackage{hyperref}
\usepackage{mathrsfs}
\usepackage{graphicx,color}
\usepackage[alphabetic]{amsrefs}
\usepackage[all,cmtip]{xy}
\usepackage{bbm}

\newtheorem{prop}{Proposition}[section]
\newtheorem{thm}[prop]{Theorem}
\newtheorem{cor}[prop]{Corollary}

\newtheorem{lem}[prop]{Lemma}

\theoremstyle{definition}
\newtheorem{que}[prop]{Question}

\newtheorem{expl}[prop]{Example}

\newtheorem*{claim*}{Claim}

\newcommand{\bP}{\mathbb{P}}
\newcommand{\bC}{\mathbb{C}}

\newcommand{\bQ}{\mathbb{Q}}
\newcommand{\bZ}{\mathbb{Z}}
\newcommand{\bN}{\mathbb{N}}

\newcommand{\bk}{\mathbbm{k}}

\newcommand{\cO}{\mathcal{O}}

\newcommand{\cF}{\mathcal{F}}

\newcommand{\fa}{\mathfrak{a}}

\newcommand{\Spec}{\mathbf{Spec}}
\newcommand{\Proj}{\mathbf{Proj}}
\newcommand{\Supp}{\mathrm{Supp}}

\newcommand{\Coef}{\mathrm{Coef}}

\newcommand{\sing}{\mathrm{sing}}

\numberwithin{equation}{section}

\title{Direct summands of klt singularities}

\author{Ziquan Zhuang}

\address{Department of Mathematics, Johns Hopkins University, Baltimore, MD 21218, USA}
\email{zzhuang@jhu.edu}

\address{Department of Mathematics, Statistics, and Computer Science, University of Illinois at Chicago, Chicago, IL 60607, USA}
\email{slyu@uic.edu}

\date{}

\begin{document}

\maketitle

\begin{abstract}
    We show that direct summands (or more generally, pure images) of klt type singularities are of klt type. As a consequence, we give a different proof of a recent result of Braun, Greb, Langlois and Moraga that reductive quotients of klt type singularities are of klt type.
\end{abstract}

\medskip

\emph{Unless otherwise stated, throughout the paper we work over an algebraically closed field $\bk$ of characteristic zero.}

\section{Introduction}

Given a morphism $Y=\Spec(B)\to X=\Spec(A)$ between affine algebraic varieties, we say that $X$ is a direct summand of $Y$ if $A$ is a direct summand of $B$ as $A$-modules. More generally, we say the morphism is pure if $A$ is a pure subring of $B$, i.e. $M\to M\otimes_A B$ is injective for any $A$-module $M$. Some interesting examples are GIT quotients by reductive groups and faithfully flat morphisms between normal varieties. The purpose of this note is to show that certain classes of singularities from the Minimal Model Program (see e.g. \cite{KM98}) are preserved under pure morphisms.

\begin{thm} \label{thm:main}
Let $f\colon Y\to X$ be a pure morphism between affine varieties. Assume that $Y$ is of klt type\footnote{This means there exists some effective $\bQ$-divisor $D$ on $Y$ such that the pair $(Y,D)$ is klt. In some literature this is also called potentially klt. More generally we say a pair $(X,\Delta)$ (where $\Delta$ is an effective $\bQ$-divisor on $X$) is of klt (resp. plt, resp. lc) type if locally there exists some effective $\bQ$-divisor $D$ such that $(X,\Delta+D)$ is klt (resp. plt, resp. lc).}. Then $X$ is also of klt type. In particular, direct summands of klt type singularities are of klt type.
\end{thm}

This refines Boutot's theorem \cite{Bou-summand-rational} that pure images of rational singularities are rational. We also prove an analogous statement in the plt case, see Theorem \ref{thm:plt}. As immediate applications, we deduce that images of klt type singularities under equidimensional morphisms are of klt type (Corollary \ref{cor:equidim image}), and give a new proof of a recent result of Braun, Greb, Langlois and Moraga \cite{BGLM-reductive-quotient} that reductive quotients of klt type singularities are of klt type.

\begin{cor} \label{cor:reductive quotient}
Let $X=\Spec(A)$ be an affine variety of klt type and let $G$ be a reductive group acting on $X$. Then the quotient $X/\!/G=\Spec(A^G)$ is also of klt type.
\end{cor}

We refer to \cite{BGLM-reductive-quotient} for several further applications of the above result.

Theorem \ref{thm:main} is suggested by positive characteristic considerations: the characteristic $p>0$ analog of klt singularities is the class of strongly $F$-regular singularities, which by definition is preserved by taking direct summands (the Hochster-Roberts theorem \cite{HR-quotient-CM} that reductive quotient singularities are Cohen-Macaulay was based on a similar circle of ideas). In fact, several special cases of Theorem \ref{thm:main}, e.g. when $X$ and $Y$ come from section rings of Mori dream spaces \cite{GOST-Fano-type-Cox-ring}, or when $X$ is $\bQ$-Gorenstein \cite{Sch-pure-image-klt}, have been established by using the connection between klt and strongly $F$-regular singularities. Our proof of Theorem \ref{thm:main} does not rely on the theory of $F$-singularities (in particular we also give a new proof of Schoutens' result \cite{Sch-pure-image-klt}), but it would still be interesting to find a proof of Theorem \ref{thm:main} through the conjecture of Schwede and Smith (see \cite{SS-F-reg-type-log-Fano}*{Section 7}) that strongly $F$-regular type singularities are also of klt type. We also expect the log canonical version of Theorem \ref{thm:main} to hold (see Question \ref{q:lc case}, c.f. \cite{BGLM-reductive-quotient}*{Question 8.5}, \cite{GM-DuBois-pure-subring}*{Corollary B}), although our current method does not seem to apply to the lc setting.

\subsection*{Acknowledgement}

The author is partially supported by the NSF Grants DMS-2240926, DMS-2234736, a Clay research fellowship, as well as a Sloan fellowship. He would like to thank Lukas Braun, Javier Carvajal-Rojas, Shiji Lyu, Linquan Ma, Joaqu\'in Moraga, Karl Schwede, Kevin Tucker and Chenyang Xu for helpful discussions and comments. He also wants to thank the anonymous referee for careful reading of the manuscript and several helpful suggestions.

\section{Proof}




\subsection{Some preliminary observations}

We start by collecting some preliminary results that will be useful in the proof of Theorem \ref{thm:main}. First we recall a criterion for purity.

\begin{lem}[\cite{HR-Frob-purity}*{Corollary 5.3}] \label{lem:purity criterion}
Let $A$ be a Noetherian ring which is subring of $B$. Then $A$ is a pure subring of $B$ if and only if $A$ is a direct summand, as an $A$-module, of every finitely generated $A$-submodule of $B$ that contains it.
\end{lem}

Eventually the goal is to reduce Theorem \ref{thm:main} to the case when $X$ is Gorenstein, as \cite{Bou-summand-rational} has shown that pure images of rational singularities remain rational, and for Gorenstein singularities being rational is equivalent to being klt. The next result will provide the basis for performing various reduction steps. 





\begin{lem} \label{lem:purity preserved}
Let $f\colon Y=\Spec(B)\to X=\Spec(A)$ be a pure morphism between affine Noetherian schemes. Let $\fa\subseteq A$ be an ideal, let $U\subseteq X$ be an open subset, let $V=f^{-1}(U)\subseteq Y$, and let $L$ be a line bundle on $U$. Then
\begin{enumerate}
    \item $A/\fa$ is a pure subring of $B/\fa B$.
    \item $H^0(\cO_U)$ is a pure subring of $H^0(\cO_V)$.
    \item $\bigoplus_{m\in\bN} H^0(U,\cO_U(mL))$ is a pure subring of $\bigoplus_{m\in\bN} H^0(V,\cO_V(mf^*L))$.
\end{enumerate}
\end{lem}

\begin{proof}
\cite{HR-quotient-CM}*{Lemma 6.2} gives (1). To see (2), note that any finitely generated $H^0(\cO_U)$-submodule $N$ of $H^0(\cO_V)=H^0(U,f_*\cO_V)$ is contained in $H^0(U,\cF|_U)$ for some coherent sheaf $\cF\subseteq f_*\cO_Y$. Since $A$ is a pure subring of $B$, by Lemma \ref{lem:purity criterion} we get a spitting map $p\colon \cF\to \cO_X$, which induces a splitting $N\subseteq H^0(U,\cF|_U)\to H^0(\cO_U)$. Another application of Lemma \ref{lem:purity criterion} then gives (2) (note that \emph{a priori} $H^0(\cO_U)$ may fail to be Noetherian, but the Noetherian condition is not needed for the ``if'' direction of Lemma \ref{lem:purity criterion}; the same remark applies to the proof of part (3) below).


For any quasi-coherent sheaf $\cF$ on $X$, write $R(\cF):=\bigoplus_{m\in\bN} H^0(U,\cF\otimes\cO_U(mL))$. By projection formula, we have $H^0(V,\cO_V(mf^*L))=H^0(U,f_*\cO_V\otimes \cO_U(mL))$. Thus the two rings in part (3) are $R(\cO_X)$ and $R(f_*\cO_Y)$. The proof of (3) then proceeds just as in (2), noting that any finitely generated $R(\cO_X)$-submodule of $R(f_*\cO_Y)$ is contained in $R(\cF)$ for some coherent sheaf $\cF\subseteq f_*\cO_Y$.
\end{proof}

\subsection{A special case}

We next consider a special case of Theorem \ref{thm:main}.

\begin{lem} \label{lem:special case - Q-Gor base}
Let $f\colon Y\to X$ be a pure morphism between affine normal varieties. Assume that $Y$ is of klt type, $K_X$ is $\bQ$-Cartier, and $f^{-1}(X_{\sing})$ has codimension at least two in $Y$. Then $X$ has klt singularities.
\end{lem}

Note that the above statement, even without the assumption that $f^{-1}(X_{\sing})$ has codimension at least two in $Y$, has been proved by Schoutens \cite{Sch-pure-image-klt} using the techniques of $F$-singularities. Here we give a direct proof. In the end, our argument will also give a new proof of Schoutens's result that does not rely on $F$-singularities.

\begin{proof}
Let $x\in X$ be a closed point. The statement is local on $X$, so after shrinking $X$ we may assume that $\cO_X(rK_X)\cong \cO_X$ where $r$ is the index of $K_X$ at $x$. Let $s\in H^0(X,\cO_X(rK_X))$ be a nowhere vanishing section, let $\pi\colon X'\to X$ be the corresponding index one cover \cite{KM98}*{Definition 5.19}, and let $Y'$ be the normalization of the main components of $X'\times_X Y$. For simplicity, we also denote the map $Y'\to Y$ by $\pi$. Let $U$ be the smooth locus of $X$, let $V=f^{-1}(U)$, and let $U'$ (resp. $V'$) be the preimage of $U$ (resp. $V$) in $X'$ (resp. $Y'$). Then $\pi$ is \'etale over both $U$ and $V$, the map $V'\to V$ is the cyclic cover corresponding to the section $f^*s\in H^0(V,\cO_V(rf^*K_U))$, and we have
\[
\pi_*\cO_{U'} = \left( \bigoplus_{m\in\bN} \cO_U(mK_U)\cdot t^m\right)/(st^r-1),\,\, \pi_*\cO_{V'} = \left( \bigoplus_{m\in\bN} \cO_V(mf^*K_U)\cdot t^m\right)/(st^r-1).
\]
By assumption, the complement of $U$ (resp. $V$) in $X$ (resp. $Y$) has codimension at least two. It follows that
\begin{align*}
    H^0(\cO_{X'}) = H^0(U,\pi_*\cO_{U'}) & = \left( \bigoplus_{m\in\bN} H^0(\cO_U(mK_U))\cdot t^m\right)/(st^r-1), \\
    H^0(\cO_{Y'}) = H^0(V,\pi_*\cO_{V'}) & = \left( \bigoplus_{m\in\bN} H^0(\cO_V(mf^*K_U))\cdot t^m\right)/(st^r-1).
\end{align*}
Combined with Lemma \ref{lem:purity preserved}(3) and Lemma \ref{lem:purity preserved}(1), we see that the morphism $Y'\to X'$ is pure. Note that $Y'$ may have several connected components. Nonetheless, by \cite{HR-quotient-CM}*{Corollary 6.12} and the fact that $\pi^{-1}(x)$ consists of a single point in $X'$, we know that after possible shrinking $X$ around $x$, there exists some component (say $Y'_1$) of $Y'$ such that $Y'_1\to X'$ is pure. As $Y'_1\to Y$ is quasi-\'etale (it is \'etale over $V$), we see that $Y'_1$ is of klt type by \cite{KM98}*{Proposition 5.20}, hence it has rational singularities by \cite{KM98}*{Theorem 5.22}. By \cite{Bou-summand-rational}, this implies that its pure image $X'$ also has rational singularities. By construction, $K_{X'}$ is Cartier, thus \cite{KM98}*{Corollary 5.24} yields that $X'$ has klt singularities. As $X'\to X$ is quasi-\'etale, another application of \cite{KM98}*{Proposition 5.20} gives the desired result that the singularities of $X$ are klt.
\end{proof}




\subsection{A klt type criterion}

To deduce Theorem \ref{thm:main} from Lemma \ref{lem:special case - Q-Gor base}, two difficulties lie ahead: the image $X$ is not $\bQ$-Gorenstein in general, and even if it is $\bQ$-Gorenstein, the preimage $f^{-1}(X_{\sing})$ of the singular locus may well contain a divisor in $Y$. We will address the latter issue at the very end: the point is that after a small birational modification one can ``throw away'' the divisorial part of $f^{-1}(X_{\sing})$ from $Y$ while retaining purity. Here we focus on the $\bQ$-Gorenstein part. The key is to use the following klt type criterion.

\begin{lem} \label{lem:klt type criterion}
Let $X$ be a normal variety and $\Delta$ an effective $\bQ$-divisor on $X$. Then $(X,\Delta)$ is of klt $($resp. plt$)$ type if and only if the following two conditions are simultaneously satisfied:
\begin{enumerate}
    \item $R(X,-r(K_X+\Delta)):=\bigoplus_{m\in\bN} \cO_X(-mr(K_X+\Delta))$ is a sheaf of finitely generated $\cO_X$-algebras for some positive integer $r$ such that $\Coef(r\Delta)\in\bZ$.
    \item The pair $(X',\Delta')$, where $X':=\Proj_X R(X,-r(K_X+\Delta))$ and $\Delta'$ is the strict transform of $\Delta$, is klt $($resp. plt$)$.
\end{enumerate}
\end{lem}

\begin{proof}
We only prove the klt case since the plt case is very similar. Since the statement is local, we may assume that $X$ is affine. By \cite{KM98}*{Lemma 6.2}, as long as $R(X,-r(K_X+\Delta))$ is finitely generated, the morphism $\pi\colon X'\to X$ is an isomorphism in codimension one and $-(K_{X'}+\Delta')=-\pi^{-1}_* (K_X+\Delta)$ is $\bQ$-Cartier and $\pi$-ample. It follows that $-(K_{X'}+\Delta')$ is ample since $X$ is affine.

First suppose that $(X,\Delta)$ is of klt type. Then $R(X,-r(K_X+\Delta))$ is finitely generated by the following Lemma \ref{lem:ample model of Weil div}. If $(X,\Delta+D)$ is klt, then its crepant pullback to $X'$ remains klt. In particular, $(X',\Delta')$ is of klt type. Since $K_{X'}+\Delta'$ is $\bQ$-Cartier, this implies that $(X',\Delta')$ is klt. 

Suppose next that (1) and (2) holds. Since $-(K_{X'}+\Delta')$ is ample and $(X',\Delta')$ is klt, by Bertini's theorem we may choose some $0\le D'\sim_\bQ -(K_{X'}+\Delta')$ such that the pair $(X',\Delta'+D')$ is klt. Note that $K_{X'}+\Delta'+D'\sim_\bQ 0$, hence $K_{X'}+\Delta'+D'=\pi^*(K_X+\Delta+D)$ for some effective $\bQ$-divisor $D$ on $X$ such that $(X,\Delta+D)$ is klt. In other words, $(X,\Delta)$ is of klt type.
\end{proof}

The following result is used in the above proof. For a more general version, see \cite{BM-Cox-iteration}.

\begin{lem} \label{lem:ample model of Weil div}
Let $X$ be a variety of klt type and let $D$ be a Weil divisor on $X$. Then there exists a small birational modification $\pi\colon X'\to X$ such that the strict transform $D'=\pi^{-1}_* D$ is $\bQ$-Cartier and $\pi$-ample. In particular, the sheaf $\bigoplus_{m\in\bN} \cO_X(mD)$ of $\cO_X$-algebras is finitely generated.
\end{lem}

\begin{proof}
This follows easily from \cite{BCHM}, see e.g. \cite{Z-mld^K-1}*{Lemma 4.7}.
\end{proof}

\subsection{Finite generation}

We thus proceed to show that the two conditions in Lemma \ref{lem:klt type criterion} are satisfied by pure images of klt type singularities. In this subsection, we deal with the finite generation part; the other condition will be verified afterwards by reducing to the special case treated in Lemma \ref{lem:special case - Q-Gor base}. We need an auxiliary result.

\begin{lem} \label{lem:Spec H^0(open in klt) still klt}
Let $Y$ be an affine variety of klt type, let $V\subseteq Y$ be an open subset, and let $L$ be a Weil divisor on $V$. Then:
\begin{enumerate}
    \item $Y_0:=\Spec(H^0(\cO_V))$ is of klt type.
    \item The algebra $R(V,L)=\bigoplus_{m\in\bN} H^0(V,\cO_V(mL))$ is finitely generated over $H^0(\cO_V)$.
    \item The birational map $V\dashrightarrow Y_0$ identifies $V$ with a big open set of $Y_0$ $($i.e. $Y_0\setminus V$ has codimension at least two$)$.
    \item If $\Delta$ is an effective $\bQ$-divisor on $Y$ such that $(Y,\Delta)$ is of klt $($resp. plt$)$ type, then so is the strict transform $(Y_0,\Delta_0)$.
\end{enumerate}
\end{lem}

\begin{proof}
We shall give a geometric construction of $Y_0$. Let $D$ be the divisorial part of $Z:=\Supp(Y\setminus V)$ (if $Z$ has codimension $\ge 2$ in $Y$ then $D=0$). By Lemma \ref{lem:ample model of Weil div}, there exists a small birational morphism $\pi\colon Y'\to Y$ such that $D'=\pi_*^{-1}D$ is $\bQ$-Cartier and $\pi$-ample. By construction, we have:
\begin{itemize}
    \item $Y'\setminus D'$ is affine ($Y'$ is projective over the affine variety $Y$, and we remove from it an ample divisor $D'$).
    \item $Y'$ is of klt type (take the crepant pullback of a klt boundary on $Y$).
    \item $\pi$ is an isomorphism over $V$ (the exceptional set is contained in $D$).
\end{itemize}
By the last property, we see that $V$ is a big open subset of $Y'\setminus D'$, hence $H^0(\cO_{Y'\setminus D'})=H^0(\cO_V)$. Combined with the first property we deduce that $Y_0\cong Y'\setminus D'$, hence (3) holds and the second property above implies (1). Similarly (4) holds. The closure of $L$ gives a Weil divisor $L_0$ on $Y_0$ such that $R(V,L)=R(Y_0,L_0)$. By Lemma \ref{lem:ample model of Weil div}, the algebra $R(Y_0,L_0)$ is finitely generated. This finishes the proof. 
\end{proof}

The following lemma gives the desired finite generation.

\begin{lem} \label{lem:R(-K_X) fg}
Let $f\colon Y\to X$ be a pure morphism between normal Noetherian affine schemes. Assume that for any open subset $V\subseteq Y$ and any Weil divisor $L$ on $V$, the section ring $R(V,L)$ is finitely generated $($as an $H^0(\cO_Y)$-algebra$)$. Then $R(X,D)$ is finitely generated $($as an $H^0(\cO_X)$-algebra$)$ for any Weil divisor $D$ on $X$.
\end{lem}

By Lemma \ref{lem:Spec H^0(open in klt) still klt}, the above assumption on $Y$ holds when $Y$ is a variety of klt type over $\bk$. Lemma \ref{lem:R(-K_X) fg} also applies when $Y$ is $\bQ$-factorial: in this case the divisorial part of $Y\setminus V$ supports some Cartier divisor $D$, and replacing $V$ with the open set $Y\setminus D$ does not change the section ring; but as $L$ is $\bQ$-Cartier and $Y\setminus D$ is affine, the section ring $R(Y\setminus D,L)$ is finitely generated.

\begin{proof}
Let $U$ be the regular locus of $X$. By Lemma \ref{lem:purity preserved}(3), $R(X,D)=R(U,D_U)$ is a pure subalgebra of $R(f^{-1}(U),f^*D_U)$, and the latter is finitely generated by our assumption. If $f$ is of finite type, it then follows from \cite{Has-pure-subalg-fg} that $R(X,D)$ is finitely generated. For the general case, denote $R:=R(X,D)$ and $S:=R(f^{-1}(U),f^*D_U)$. Note that $R$ is a \emph{graded} pure subalgebra of $S$ and its degree $0$ part is exactly $H^0(\cO_X)$. Let $I=\bigoplus_{m\ge 1} R_m:=\bigoplus_{m\ge 1} H^0(X,mD)$. Since $S$ is Noetherian, the ideal $IS$ is also finitely generated, hence we may write $IS=(g_1,\dots,g_m)S$ for some homogeneous element $g_1,\dots,g_m\in I$. By purity, the map 
\[
R/(g_1,\cdots,g_m)R \to S/(g_1,\cdots,g_m)S = S/IS
\]
is injective. On the other hand its kernel contains $I/(g_1,\cdots,g_m)R$ by construction, thus $I=(g_1,\cdots,g_m)R$. In other words, any homogeneous element $g\in R$ of positive degree can be written as $g=g_1 h_1+\dots+g_m h_m$ for some $h_1,\dots,h_m\in R$. By induction on the degree, this implies that $R=R_0[g_1,\dots,g_m]$, hence $R$ is finitely generated.
\end{proof}


\subsection{Completion of the proof}

Returning to the setup of Theorem  \ref{thm:main}, we now let $X':=\Proj_X(R(X,-K_X))$. Recall that $X'\to X$ is a small modification, and $-K_{X'}$ is $\bQ$-Cartier and ample. It remains to show:

\begin{lem} \label{lem:Q-Gor model klt}
$X'$ has klt singularities.
\end{lem}

\begin{proof}
Let $m$ be a sufficiently divisible positive integer. Let $0\le D'\sim -mK_{X'}$ and let $U_1=X'\setminus D'$. Then $U_1$ is an affine open subset of $X'$ and as we vary $D'$, the corresponding $U_1$ cover $X'$. Thus it suffices to show that $U_1$ has klt singularities. Let $D$ be the strict transform of $D'$ on $X$, let $U=X\setminus (X_{\sing}\cup D)$, and let $V=f^{-1}(U)\subseteq Y$. Then $U$ is a big open subset of $U_1$, hence $H^0(\cO_{U_1})=H^0(\cO_U)$. By Lemma \ref{lem:purity preserved}(2), this is a pure subring of $H^0(\cO_V)$. Consider the induced morphism $V_1:=\Spec(H^0(\cO_V))\to U_1$. By Lemma \ref{lem:Spec H^0(open in klt) still klt}, $V_1$ is of klt type and $V_1\setminus V$ has codimension at least two. By construction, $K_{U_1}$ is $\bQ$-Cartier and the singular locus of $U_1$ is contained in $U_1\setminus U$. Thus the assumptions of Lemma \ref{lem:special case - Q-Gor base} are satisfied and we deduce that $U_1$ is klt as desired.
\end{proof}

\begin{proof}[Proof of Theorem \ref{thm:main}]
By \cite{HR-quotient-CM}*{Proposition 6.15}, $X$ is normal. The result then follows from Lemma \ref{lem:klt type criterion}, Lemma \ref{lem:R(-K_X) fg} and Lemma \ref{lem:Q-Gor model klt}.
\end{proof}

\begin{cor} \label{cor:equidim image}
Let $f\colon Y\to X$ be an equidimensional surjective morphism between affine varieties. Assume that $Y$ is of klt type. Then $X$ is also of klt type.
\end{cor}

Recall that a surjective morphism $f\colon Y\to X$ between algebraic varieties is equidimensional if every fiber has the same dimension (which necessarily equals $\dim Y-\dim X$).

\begin{proof}
By \cite{CGM-pure-subring}*{Lemma 2.7}, $f$ is pure. The result then follows from Theorem \ref{thm:main}. Alternatively, by Bertini's theorem and taking hyperplane sections one reduces to the case when $f$ is quasi-finite, then the result follows from \cite{FG-cbf}*{Lemma 1.1}.
\end{proof}

\begin{proof}[Proof of Corollary \ref{cor:reductive quotient}]
Since $G$ is reductive and $\mathrm{char}(\bk)=0$, $A^G$ is a direct summand of $A$, thus the result follows from Theorem \ref{thm:main}.
\end{proof}

\subsection{The plt case}

With a slight modification of the previous argument, we can also prove an analogous statement in the plt case \footnote{Thanks to Javier Carvajal-Rojas for suggesting this refinement.}.

\begin{thm} \label{thm:plt}
Let $f\colon Y\to X$ be a pure morphism between affine varieties, let $P$ be a prime divisor on $X$, and let $Q=f^* P$ be the cycle-theoretic pullback. Assume that $(Y,Q)$ is of plt type. Then $(X,P)$ is of plt type.
\end{thm}

Here we define the cycle-theoretic pullback $f^*P$ of a prime divisor $P$ as the Weil divisor associated to the divisorial part of the scheme-theoretic pullback $f^{-1}(P)$. For general $\bQ$-divisors $D=\sum a_i D_i$, the cycle-theoretic pullback is defined by linearity, i.e. $f^*D:=\sum a_i f^*D_i$.

\begin{proof}
By replacing $Y$ with $\Spec(H^0(\cO_V))$ where $V=Y\setminus f^{-1}(X_{\sing})$ and applying Lemma \ref{lem:Spec H^0(open in klt) still klt}, we may assume throughout the proof that $f^{-1}(X_{\sing})$ has codimension at least two in $Y$. In this case, the cycle-theoretic pullback $f^*P$ is determined by its restriction to $V$, where it becomes the pullback of a Cartier divisor, and hence commutes with finite base change over $X$. By \cite{KM98}*{Proposition 2.43}, we know that $Y$ is of klt type, thus so is $X$ by Theorem \ref{thm:main}. In particular, $X$ is normal. By Lemma \ref{lem:ample model of Weil div}, there exists a small birational morphism $X'\to X$ such that $-(K_{X'}+P')$ is $\bQ$-Cartier and ample over $X$, where $P'$ is the strict transform of $P$. By Lemma \ref{lem:klt type criterion}, it suffices to show that $(X',P')$ is plt. As in the proof of Lemma \ref{lem:Q-Gor model klt}, using Lemma \ref{lem:purity preserved} and Lemma \ref{lem:Spec H^0(open in klt) still klt} we can cover $X'$ by affine open sets $U_i$ such that each $U_i$ is the pure image of some affine variety $V_i$ where $V_i\to U_i$ satisfies the same assumptions of the theorem. Thus replacing $X$ by the various $U_i$'s we may assume that $K_X+P$ is $\bQ$-Cartier. 

By Lemma \ref{lem:ample model of Weil div} again, there also exists a small birational morphism $X''\to X$ such that the strict transform $P''$ of $P$ is $\bQ$-Cartier and ample over $X$. Note that $(X,P)$ is plt if and only if $(X'',P'')$ is plt, since the morphism is small and hence crepant. As before, $X''$ admits affine open covers of the form $X''\setminus D''$ where $0\le D''\sim mP''$ for some integer $m$, and each of these affine subsets can be realized as the image of some pure morphism that satisfies the assumptions of the theorem (again this follows from Lemma \ref{lem:purity preserved} and Lemma \ref{lem:Spec H^0(open in klt) still klt}, as in the proof of Lemma \ref{lem:Q-Gor model klt}). Thus replacing $X$ by these affine subsets of $X''$, we can further assume that $P$ is also $\bQ$-Cartier.

Consider the index one cover $X_1\to X$ for $P$ and let $P_1$ be the preimage of $P$. Recall from the beginning of the proof that $f^{-1}(X_{\sing})$ has codimension at least two in $Y$. Thus as in the proof of Lemma \ref{lem:special case - Q-Gor base}, (possibly after shrinking $X$) there exists some component $Y_1$ of the normalization of $X_1\times_X Y$ such that the induced morphism $f_1\colon Y_1\to X_1$ is pure and $(Y_1,Q_1)$ is of plt type. Here $Q_1$ is the preimage of $Q$; it coincides with the cycle-theoretic pullback $f_1^* P_1$ as $f^{-1}(X_{\sing})$ has codimension at least two in $Y$. Now $P_1$ is Cartier by construction, and $(X,P)$ is plt if and only if $(X_1,P_1)$ is plt. Replacing $(X,P)$ with $(X_1,P_1)$, we reduce to the case when $K_X+P$ is $\bQ$-Cartier and $P$ is Cartier.

Since $P$ is Cartier, the induced morphism $Q\to P$ is pure by Lemma \ref{lem:purity preserved}(1), and $Q$ is of klt type by adjunction. Thus $P$ is klt by Theorem \ref{thm:main}. In particular, it is normal. As we also have that $K_X+P$ is $\bQ$-Cartier, inversion of adjunction \cite{KM98}*{Theorem 5.50} implies that $(X,P)$ is plt. This finishes the proof. 
\end{proof}

\subsection{Some remarks and further questions} 
A natural question is whether Theorem \ref{thm:main} generalizes to the log canonical setting.

\begin{que} \label{q:lc case}
Let $f\colon Y\to X$ be a pure morphism between affine varieties. Assume that $Y$ is of lc type. Is $X$ also of lc type?
\end{que}

This question has a positive answer if $K_X$ is Cartier \cite{GM-DuBois-pure-subring}*{Corollary B}. One of the difficulties in extending our arguments to the general lc case is that Lemma \ref{lem:klt type criterion} fails if we replace klt (or plt) by lc, as illustrated by the following example.

\begin{expl}
We show that there exist lc type singularities $X$ for which $R(X,-K_X)$ is not finitely generated. Let $C_0\subseteq \bP^2$ be a smooth elliptic curve, and let $P_1,\dots,P_9$ be nine very general points on $C_0$ (we may assume $\bk=\bC$). Let $S$ be the blowup of $\bP^2$ at $P_1,\dots,P_9$, let $C\subseteq S$ be the strict transform of $C_0$, and let $X$ be the cone over $S$ (with respect to some sufficiently ample divisor $H$). Then $K_S+C\sim 0$ and $(S,C)$ is log smooth, which implies the cone over $(S,C)$ is lc \cite{Kol13}*{Lemma 3.1}, thus $X$ is of lc type. We have 
\[
R(X,-K_X)\cong \bigoplus_{\ell,m\in \bN} H^0(S,\ell C+mH).
\]
Since the $P_i$'s are very general, $\cO_C(C)$ is a non-torsion line bundle of degree zero. It follows that $h^0(C,\ell C)=0$ and $h^0(C,\ell C+H)>0$ for all $\ell \in \bN$. Since $H$ is sufficiently ample we also have $h^1(S,\ell C+H)=0$ by Fujita vanishing. Using the two exact sequences 
\[
0\to \cO_S((\ell-1)C)\to \cO_S(\ell C)\to \cO_C(\ell C)\to 0,
\]
\[
0\to \cO_S((\ell-1)C+H)\to \cO_S(\ell C+H)\to \cO_C(\ell C+H)\to 0,
\]
we see that $\ell C$ is the only member of the linear system $|\ell C|$, and $h^0(S,\ell C+H)>h^0(S,(\ell-1)C+H)$. They together imply that a new generator is need for $H^0(S,\ell C+H)$ for each $\ell$ and therefore $R(X,-K_X)$ cannot be finitely generated.
\end{expl}

Another natural question is whether Theorem \ref{thm:main} and Theorem \ref{thm:plt} generalizes to pairs with more general coefficients.

\begin{que} \label{q:pairs}
Let $f\colon Y\to X$ be a pure morphism between affine varieties, let $\Delta$ be an effective $\bQ$-divisor on $X$, and let $\Delta_Y=f^* \Delta$ be the cycle-theoretic pullback. Assume that $(Y,\Delta_Y)$ is of klt (resp. lc) type. Is $(X,\Delta)$ also of klt (resp. lc) type?\footnote{Postscript note: the klt case of this question has been answered by \cite{TY-adj-ideal}*{Corollary 1.3} in the positive. They also give a partial affirmative answer \cite{TY-adj-ideal}*{Theorem 1.4} for the log canonical case.}
\end{que}

It would be interesting to know if there is some local version of the canonical bundle formula for pure morphisms. If so, it may provide answers to both Question \ref{q:lc case} and Question \ref{q:pairs}.

\appendix

\section{Pure images of klt type excellent schemes}

\centerline{Shiji Lyu}

\medskip

In this appendix, we explain how to extend the previous results to morphisms between excellent schemes admitting dualizing complexes, and slightly further. 

\subsection{Main theorems for non-finite-type schemes}
In this subsection we prove Theorems \ref{thm:main} and \ref{thm:plt}
for excellent schemes admitting dualizing complexes.

Let $X$ be a Noetherian excellent scheme of equal characteristic zero that admits a dualizing complex. 
We say $(X,\Delta)$ is \emph{of klt (resp. plt) type}
if Zariski locally on $X$ there exists a $\mathbb{Q}$-divisor $D\geq 0$ such that $(X,\Delta+D)$ is klt (resp. plt).

Since log resolutions exist \cite{Tem08-resol}*{Theorem 2.3.6},
being klt or plt can be detected using a single log resolution,
thus
$(X,\Delta)$ is {of klt (resp. plt) type}
if and only if for all $x\in X$, $(\Spec(\cO_{X,x}),\Delta|_{\Spec(\cO_{X,x})})$ is {of klt (resp. plt) type}.

We say that a morphism of Noetherian schemes $f:Y\to X$ is \emph{pure}
if for all $x\in X$,
there exists $y\in Y$ such that $f(y)=x$ and
$\cO_{X,x}\to\cO_{Y,y}$ is pure.
If $Y=\Spec(B)$ and $X=\Spec(A)$ are affine, $f$ is pure if and only if $A\to B$ is pure,
see \cite{HH95-lemmaforpure}*{Lemma 2.2}.

Here is our extension of Theorems \ref{thm:main} and \ref{thm:plt}.
\begin{thm}\label{thm:klt plt excellent}
Let $f:Y\to X$ be a pure morphism between Noetherian schemes of equal characteristic zero.
Assume that both $X$ and $Y$ are excellent and admit dualizing complexes. 
Then the followings hold.
\begin{enumerate}
    \item Assume that $Y$ is of klt type. Then $X$ is also of of klt type.
    \item Let $P$ be a prime divisor on $X$, and let $Q$ be the divisorial part of the scheme-theoretic pullback $f^{-1}(P)$.
    If $(Y,Q)$ is of plt type, then $(X,P)$ is also of plt type.
\end{enumerate}
\end{thm}

We also remark that Corollary \ref{cor:equidim image} can be extended as well.

\begin{cor}\label{cor:equidimensional excellent}
Let $f:Y\to X$ be an equidimensional morphism of finite type between Noetherian schemes of equal characteristic zero. 
Assume that $X$ is normal and excellent, and that $X$ admits dualizing complexes.
%
%
Then the followings hold.
\begin{enumerate}
    \item Assume that $Y$ is of klt type. Then $X$ is also of of klt type.
    \item Let $P$ be a prime divisor on $X$, and let $Q$ be the divisorial part of the scheme-theoretic pullback $f^{-1}(P)$.
    If $(Y,Q)$ is of plt type, then $(X,P)$ is also of plt type.
\end{enumerate}
\end{cor}
\begin{proof}
A finite type map preserves being excellent and admitting dualizing complexes.
Thus it suffices to show $f$ is pure.

By \cite{EGA4_3}*{Proposition 13.3.1}, $f$ factors locally as
$Y\xrightarrow{g}\mathbb{A}^{e}_X\to X$
where $g$ is quasi-finite and $e=\dim Y-\dim X$.
Since $\mathbb{A}^{e}_X$ is normal and of equal characteristic zero, $g$ is pure, thus so is $f$.
\end{proof}

We cannot readily extend Corollary \ref{cor:reductive quotient}, 
since we do not know if excellence and dualizing complexes can be descended from $A$ to $A^G$.
However, see Corollary \ref{cor:reductive quotient general} below.\\

We now turn to the proof of Theorem \ref{thm:klt plt excellent}.
We will not give all the details in our case, since the argument will be completely parallel to the proof of Theorems \ref{thm:main} and \ref{thm:plt}.
We will only indicate which parts of the argument need to be modified in our situation.

In the proof of Lemma \ref{lem:special case - Q-Gor base},
\cite{KM98}*{Proposition 5.20} works with no problem in our case, and the reference \cite{Bou-summand-rational}
can be replaced by \cite{Mur-vanishing}*{Theorem C}.
Proper birational map from a regular scheme satisfies
Grauert-Riemenschneider \cite{Mur-vanishing}*{Theorem A}.
Thus one can apply the argument in \cite{Kol97-rational}*{\S 11}
to see that \cite{KM98}*{Theorem 5.22 and Corollary 5.24} hold in our case.
%
%

In the proof of Lemma \ref{lem:klt type criterion} (for klt type),
\cite{KM98}*{Lemma 6.2} works with no problem,
and the required Bertini theorem is \cite{LM-MMP}*{Corollary 10.4}. 
We still need to prove Lemma \ref{lem:ample model of Weil div}.
We follow the proof of \cite{Z-mld^K-1}*{Lemma 4.7}.
The existence of $\mathbb{Q}$-factorialization
is \cite{LM-MMP}*{Corollary 22.3}.
Small perturbation of a klt pair is klt \cite{LM-MMP}*{Lemma 6.9},
and the log canonical model exists due to the finite generation of relative adjoint rings \cite{LM-MMP}*{Theorem 17.3}.

The proof of Lemmas \ref{lem:Spec H^0(open in klt) still klt}, \ref{lem:R(-K_X) fg} and \ref{lem:Q-Gor model klt} works verbatim.
At this point, we have proved
statement (1) of Theorem \ref{thm:klt plt excellent}.
%
%



Let us now consider statement (2).
Again, we follow the proof of Theorem \ref{thm:plt}.
The first step is to show $Y$ of klt type, so $X$ of klt type.
It suffices to prove \cite{KM98}*{Proposition 2.43} for every affine local excellent scheme $X$ 
of equal characteristic zero that admits a dualizing complex.
We can, for simplicity, assume $H=0$.
We use the same argument as in \cite{KM98},
but we write out much of the details
since we do not have a Bertini theorem stated in \cite{LM-MMP} for the linear system of a Weil divisor.
Since our $X$ is local and excellent,
by Hironaka's resolution of singularities,
there is a log resolution
$\pi:X'\to X$ that is an iterated blow-up of regular centers disjoint from $X\setminus Z$.
In particular, there is a $\pi$-ample $\pi$-exceptional Cartier divisor $H'(\leq 0)$ on $X'$.
Then $aH'+\pi^{-1}_*m\Delta_1$ is $\pi$-generated \cite{LM-MMP}*{Definition 4.1} for some positive integer $a$.
Some member $D'\in |aH'+\pi^{-1}_*m\Delta_1|$
then satisfies $(X',\pi^{-1}_*\Delta+D')$ snc
by \cite{LM-MMP}*{Theorem 10.1 and Remark 10.2},
so $D:=\pi_*D'$
is such that $D\sim m\Delta_1$ and that $(X\setminus Z,(\Delta+D)|_{X\setminus Z})$ is snc.
We can then argue as in the second last paragraph of the proof of
\cite{KM98}*{Proposition 2.43}. 

We now need Lemma \ref{lem:klt type criterion} for plt type.
Since we know \cite{KM98}*{Proposition 2.43} in our case,
we can apply the same proof for the klt case,
except for the Bertini theorem.
\cite{LM-MMP}*{Corollary 10.4} holds for plt instead of klt by a similar proof,
which is what we want.
(Alternatively, one can use inversion of adjunction as noted below.)

Now we can follow the argument until the last paragraph.
For the last step of the proof, inversion of adjunction holds in our case since resolutions exist and the connectedness theorem \cite{KM98}*{Theorem 5.48} holds, the latter depending (only) on Kawamata–Viehweg vanishing \cite{Mur-vanishing}*{Theorem A}.
We have now proved Theorem \ref{thm:klt plt excellent}.

\subsection{Schemes formally of klt or plt type}

For a general Noetherian scheme $X$ of equal characteristic zero,
we say $(X,\Delta)$ is \emph{formally of klt (resp. plt) type}
if for all $x\in X$,
the completion $A=\cO_{X,x}^\wedge$ is normal and
$(\Spec(A),\Delta|_{\Spec(A)})$ is of klt (resp. plt) type. 
Here, for a prime divisor $P$ on $X$,
$P|_{\Spec(A)}$ is the divisorial part of $P\times_X \Spec(A)$
and  $\Delta|_{\Spec(A)}$ is defined by linearity.
We note that $\Spec(A)$ is excellent and admits a dualizing complex since $A$ is complete.

If $X$ is excellent and admits a dualizing complex, and $(X,\Delta)$ is of klt (resp. plt) type, 
then $(X,\Delta)$ is formally of klt (resp. plt) type 
(cf. \cite{Kol13}*{Proposition 2.15}).
The converse holds when $\Delta=0$ (resp. $\Delta$ is a prime divisor) 
by our Theorem \ref{thm:klt plt excellent}.\footnote{The author thanks Zhiyuan Chen for pointing this out to me.
The author does not know if a general pair $(X,\Delta)$ being formally of klt (resp. plt) type
implies $(X,\Delta)$ being of klt (resp. plt) type, 
even when $X$ is excellent and admits a dualizing complex.
}

Now, if $A\to B$ is a pure local map of Noetherian local rings,
then
by \cite{Fed83-complete}*{Proposition 1.3(5)}
it is clear that $A^\wedge\to B^\wedge$ is pure.
Let $P$ is a prime divisor on $\Spec(A)$ and let $Q$ (resp. $P^\wedge$, $Q^\wedge$) be 
the divisorial part of its scheme-theoretic pullback to $\Spec(B)$ (resp. $\Spec(A^\wedge)$, $\Spec(B^\wedge)$).
If $(\Spec(B^\wedge),Q^\wedge)$ is of plt type,
then $Q^\wedge$ must be a prime divisor, thus so is $P^\wedge$.
Therefore Theorem \ref{thm:klt plt excellent} gives
\begin{thm}\label{thm:formal klt or plt}
Let $f:Y\to X$ be a pure morphism between Noetherian schemes of equal characteristic zero. 
Then the followings hold.

\begin{enumerate}
    \item Assume that $Y$ is formally of klt type. Then $X$ is also formally of of klt type.
    \item Let $P$ be a prime divisor on $X$, and let $Q$ be the divisorial part of the scheme-theoretic pullback $f^{-1}(P)$.
    If $(Y,Q)$ is formally of plt type, then $(X,P)$ is also formally of plt type.
\end{enumerate}
\end{thm}

Now we are able to extend Corollary \ref{cor:reductive quotient}.
\begin{cor}\label{cor:reductive quotient general}
Let $k$ be a field of characteristic zero, $G$ a reductive $k$-group, $A$ a Noetherian $k$-algebra that admits a $k$-rational $G$-action.
Then the followings hold.
\begin{enumerate}
    \item Assume that $\Spec(A)$ is formally of klt type. Then $\Spec(A^G)$ is also formally of klt type.
    \item Let $P$ be a prime divisor on $\Spec(A^G)$, and let $Q$ be the divisorial part of the scheme-theoretic pullback $P\times_{\Spec(A^G)}\Spec(A)$.
    If $(\Spec(A),Q)$ is formally of plt type, then $(\Spec(A^G),P)$ is also formally of plt type.
\end{enumerate}

\end{cor}
\begin{proof}
The map $A^G\to A$ is pure, in fact split, see for example \cite{HR-quotient-CM}*{\S 10}.
Thus $A^G$ is Noetherian (cf. the proof of Lemma \ref{lem:R(-K_X) fg}) 
and Theorem \ref{thm:formal klt or plt} applies.
\end{proof}

Again, we have a version of Corollary \ref{cor:equidim image}.
The proof is the same as Corollary \ref{cor:equidimensional excellent}.
\begin{cor}\label{cor:equidimensionalnon excellent}
Let $f:Y\to X$ be an equidimensional surjective morphism of finite type between Noetherian schemes of equal characteristic zero. 
Assume that $X$ is normal. 
%
%
Then the followings hold.
\begin{enumerate}
    \item Assume that $Y$ is formally of klt type. Then $X$ is also formally of klt type.
    \item Let $P$ be a prime divisor on $X$, and let $Q$ be the divisorial part of the scheme-theoretic pullback $f^{-1}(P)$.
    If $(Y,Q)$ is formally of plt type, then $(X,P)$ is also formally of plt type.
\end{enumerate}
\end{cor}

\bibliography{ref}

\end{document}